\documentclass[11pt]{article}
\usepackage[utf8]{inputenc}
\usepackage{amsmath,amscd}
\usepackage{amsfonts}
\usepackage{amssymb}
\usepackage{amsthm}
\usepackage{mathrsfs}
\usepackage{lscape}
\usepackage{euscript}
\usepackage{latexsym,enumerate,color,mdwlist}
\usepackage{tikz}
\usepackage{a4wide}
\usepackage{hyperref}

\usepackage{mathtools}

\usepackage[shortlabels]{enumitem}
\setlist[enumerate]{leftmargin=25pt}
\setlist[itemize]{leftmargin=25pt}

\allowdisplaybreaks

\theoremstyle{plain} \newtheorem{Thm}{Theorem}[section]

\newtheorem{lem}[Thm]{Lemma}

\newtheorem{cor}[Thm]{Corollary}

\theoremstyle{definition}

\newtheorem*{theorem*}{Theorem}
\theoremstyle{remark}

\newcommand{\comment}[1]{}

\def\N{\mathbb{N}}
\def\Q{\mathbb{Q}}
\def\Z{\mathbb{Z}}

\def\R{\mathbb{R}}
\def\La{\Lambda}
\def\la{\lambda}

\def\stb{,\ldots ,}

\def\stb{,\ldots ,}

\newcommand\blfootnote[1]{%
  \begingroup
  \renewcommand\thefootnote{}\footnote{#1}%
  \addtocounter{footnote}{-1}%
  \endgroup
}

\title{Fuglede's conjecture holds on
$\Z_p^2 \times \Z_q$}
\author{Gergely Kiss \thanks{Alfr\'ed R\'enyi Institute of Mathematics, e-mail: kigergo57@gmail.com}
\qquad
G\'abor Somlai \thanks{E\"otv\"os Lor\'and University, Faculty of Science, Institute of Mathematics,
 	e-mail: zsomlei@caesar.elte.hu}}
\begin{document}
\maketitle

\begin{abstract} The study of Fuglede's conjecture on the direct product of elementary abelian groups was initiated by Iosevich et al. For the product of two elementary abelian groups the conjecture holds. For $\Z_p^3$ the problem is still open if $p\ge 11$. In connection we prove that Fuglede's conjecture holds on $\Z_{p}^2\times \Z_q$ by developing a method based on ideas from discrete geometry. \end{abstract}

\blfootnote{%
    }

\section{Introduction}
Let $\Omega$ be a bounded measurable set with positive Lebesgue measure in $\R^n$. We say that $\Omega$ is a {\it tile} if there exists $T \subset \R^n$ such that for almost every element $x$ of $\R^n$ we can uniquely write $x$ as the sum of an element of $\Omega$ and $T$.
In this case $T$ is called the {\it tiling complement} of $\Omega$.
We say that $\Omega$ is {\it spectral} if there is a base of $L^2(\Omega)$ consisting only of exponential functions $\{f(x)=e^{2\pi i<x,\la>}|\lambda\in \Lambda\}$. In this case $\La$ is called a {\it spectrum} for $\Omega$.
The classical Fuglede conjecture \cite{F1974} states that the spectral sets are the tiles in $\R^n$.

The conjecture was motivated by the result of Fuglede \cite{F1974} that if $\Omega$ is a tile with a tiling complement, which is a lattice, then $\Omega$ is spectral.
After some valuable positive results, Tao \cite{T2004} disproved the conjecture by showing a spectral set which is not a tile in $\mathbb{R}^n$ for $n\ge 5$. This was improved in two ways. Firstly, there were found some non-tiling spectral sets
in $\R^n$ for $n\ge 4$ in \cite{Matolcsi2005} and later $n\ge 3$ in \cite{KMkishalmaz}.
Secondly, there were shown non-spectral tiles in $\mathbb{R}^n$ for $n\ge 3$ \cite{FMM2006} (for further references see \cite{FR2006, KM2}).
All of these counterexamples are based on {\bf Fuglede's conjecture on finite Abelian groups}. We remarkably note that both directions of the conjecture are still open in $\R$ and $\R^2$.

Let $G$ be a finite Abelian group and $\widehat{G}$ the set of irreducible representations of $G$, which can be considered as a group and it is isomorphic to $G$. The elements of $\widehat{G}$ are indexed by the elements of $G$. Then $S\subset G$ is {\it spectral} if and only if there exists a $\La\in G$ such that ($\chi_l)_{l\in \La}$ is an orthogonal base of complex valued functions defined on $S$. It is worth to note that if $\La$ is a spectrum for $S$, then $S$ is a spectrum for $\La$, and we say that $(S,\La)$ is a {\it spectral pair}. It follows simply that $|S|=|\La|$. On the other hand, if $S$ and $\La$ satisfy the following conditions: $|S|=|\La|$ and the elements of $(\chi_l)_{l\in \La}$ are pairwise orthogonal on $S$, then $(S,\La)$ is a spectral pair.

For a finite group $G$ and a subset $S$ of $G$ we say that $S$ is {\it a tile} of $G$ if there is a $T \subset G$ such that $S+T=G$ and $|S|\cdot |T|=|G|$. This we denote by $S\bigoplus T$. The discrete version of the problem is the following: which finite abelian groups satisfy that the spectral sets and the tiles coincide.

The case of finite cyclic groups is particularly interesting since Dutkay and Lai \cite{DutkayLai} showed that the tile-spectral direction of Fuglede's conjecture on $\R$ holds if and only if the discrete version holds for every finite cyclic group.
They also showed that if the spectral-tile direction holds on $\R$, then it holds for every finite abelian group.

Tao's \cite{T2004} example for a spectral set which does not tile comes from a non-tiling spectral set in $\Z_3^5$ so it makes sense to investigate elementary abelian $p$-groups. It was proved in \cite{Z_p^2} that for $\Z_p^2$ Fuglede's conjecture holds for every prime $p$. On the other hand, it was shown in \cite{atenetal}
that the spectral tile direction of the conjecture does not hold for $\Z_p^5$ if $p$ is an odd prime and for $\Z_p^4$ if $p \equiv 3 \pmod{4}$ is a prime.
This was strengthened by Ferguson and Sothanaphan by exhibiting a non-spectral tile for $\Z_p^4$, see \cite{fergsoth2019}.
If $p=2$, then the situation is slightly different. It was shown that Fuglede's conjecture fails for $\Z_2^d$ if $d\ge 10$ \cite{fergsoth2019}, and holds if $d\le 6$ \cite{fergsoth2019,fsadd}. For $7\le d\le 9$ the answer is not known.

The question whether Fuglede's conjecture holds for $\Z_p^3$, when $p$ is an odd prime, is still widely open, although partial results have been obtained recently for $p \le 5$  in \cite{birklbauer} and for $p \le 7$ in \cite{zp37} .
Note that the tile-spectral conjecture holds for $\Z_p^3$, see \cite{atenetal}.

R. Shi investigated the mixed direct products. He verified the conjecture for $\Z_{p^2} \times \Z_p$, recently. Reaching closer to decide the validity of Fuglede's conjecture for $\Z_p^3$, in our paper we prove the following.  
\begin{Thm}\label{thm1}
Fuglede's conjecture holds on $\mathbb{Z}_p^2 \times \mathbb{Z}_q$, where $q$ and $p$ are different primes.
\end{Thm}

\section{Irreducible Representations of \texorpdfstring{$\mathbb{Z}_p^2 \times \mathbb{Z}_q$}{}}
Let $G=\mathbb{Z}_p^2 \times \mathbb{Z}_q$. We can consider the elements of $G$ as pairs $(u,v)$, where $u \in \Z_p^2$ and $v \in \Z_q$. We call $v$ the $q$-coordinate of $(u,v)$.
The irreducible representations of $G$ are indexed by the elements of $G$ so for every $(a,b)\in \Z_p^2\times\Z_q$ let the character $\chi_{(a,b)}$ be defined as
\begin{equation}\label{eqeki}
\chi_{(a,b)}(u,v)=e^{\frac{2\pi i}{p}\langle u,a\rangle}e^{\frac{2\pi i}{q}v\cdot b} ~~~((u,v)\in \Z_p^2\times\Z_q),
\end{equation}
where $\langle u,a\rangle$ denotes the natural scalar product of the vectors $u,a \in \Z_p^2$. This description is in accord with the well-known fact that for every abelian group $H$ the set of irreducible representations  (denoted by $\widehat{H}$) is isomorphic to $H$.

We denote the order of $g\in G$ by $o(g)$.
We can distinguish $4$ different types of irreducible representations according to their order. Namely, $o(\chi) \in \{1,p,q,pq\}$.
The trivial representation is the only one of order $1$. According to the parametrization given in \eqref{eqeki}, the characters of order $p$ are precisely those of the form $\chi_{(a,0)}$, where $a \in \Z_p^2\setminus\{0\}$. Similarly, the characters of order $q$ can be written as $\chi_{(0,b)}$ with $b \in \Z_q\setminus\{0\}$. Finally, the characters of order $pq$ are the $\{ \chi_{(a,b)}  \mid (a,b)\in G,~ a \ne 0, ~ b \ne 0\}.$ Now we describe the kernel of the nontrivial irreducible representations.

\noindent
{\bf \underline{Equidistributed property:}} Let $M=\{(u_i,v_i) \mid i \in I \}$
 be a multiset (i.e, some $(u_i, v_i)$ may coincide) and $\chi_{(a,0)} \in \widehat{G}$ of order $p$. Assume $M$ is in the kernel of $\chi_{(a,0)}$. This means that $$\chi_{(a,0)}(M):=\sum_{i \in I} \chi_{(a,0)}\big(\left(u_i,v_i\right)\big)=\sum_{i \in I} e^{\frac{2\pi i}{p}\langle u_i,a\rangle}=0.$$
The minimal polynomial of $e^{\frac{2\pi i}{p}}$ over $\Q$ is $ \sum_{j =0}^{p-1}x^j$. It implies that  $|M_k|=|\{ (u,v) \in M \mid \langle u,a\rangle \equiv k \pmod{p}\}|=\frac{|M|}{p}$
 for each $k=0 \stb p-1$. In this case we say that $M$ is {\it equidistributed} on the cosets of the subgroup $\{(u,v) \in \Z_p^2 \times \Z_q \mid \langle u,a\rangle=0\}$. 

Let $b \in \Z_q \setminus \{0\}$. A similar argument as before shows that $\chi_{(0,b)}(M)=0$ if and only if $M$ is equidistributed on the cosets of $\Z_p^2$, which is the unique subgroup of $\Z_p^2\times \Z_q$ of order $p^2$. Hence $|\{(u,v) \in M \mid v=l \}|=\frac{|M|}{q}$ for every $l \in \Z_q$.
\begin{cor}\label{cor1}
Let $\chi$ be an irreducible representation of $G$ and $M$ a multiset on $G$. If $\chi(M)=0$ and $o(\chi)=p$ (resp. $o(\chi)=q$), then $p \mid |M|$ (resp. $q \mid |M|$).
\end{cor}
Finally, let $\chi_{(a,b)}$ be an irreducible representation of $G$ of order $pq$ and $M$ a multiset on $G$. Assume $\chi_{(a,b)}(M)=0$. Let $H$ be a subgroup isomorphic to $\Z_{pq}$ generated by $(a,0)$ and $(0,b)$, where $a \ne 0$ and $b \ne 0$. We define the {\it projection} $M_{(a,b)}$ of $M$ on $H$ as follows
 \begin{equation}\label{eqproj} M_{(a,b)}((u,v)) = |\{(x,v) \in M \mid  \langle x-u,a\rangle=0 \}|, \end{equation} where $u$ is a multiple of $a \in \Z_p^2$ and $v\in\Z_q$. The multiset $M_{(a,b)}$ can be considered as a nonnegative integer valued function on $\Z_{pq}$.

The lemma below is due to Lam and Leung \cite{La}. It can also be deduced from Theorem 4.9 in \cite{mi}.
\begin{lem}\label{twoprimes}
If $\chi_{(a,b)}(M)=0$, then the multiset $M_{(a,b)}$ is the weighted sum of $\Z_{p}$- and $\Z_{q}$-cosets with nonnegative integer coefficients. Thus the cardinality of $S$ satisfies $|S|=pk+q\ell$, for some $k,\ell \geq 0$.
\end{lem}

\section{Tile-Spectral}

The purpose of this section is to prove that every tile in $\Z_p^2\times \Z_q$ is a spectral set.

Let $A \bigoplus B=G$ be a tiling of $G$. Without loss of generality we assume $0 \in A \cap B$. Note that since the elements of $A+B$ are different, it follows $A \cap B=\{ 0\}$. We denote by ${\bf 1}_H$ the characteristic function of the set $H$ and by $\widehat{f}$ the Fourier transform of $f$. Then  $\widehat{{\bf 1}}_A\cdot\widehat{{\bf 1}}_B=|G|\cdot\delta_{1}$, where $\delta$ denotes the Dirac delta and $1$ denotes the trivial representation of $G$. The following is well-known.
\begin{lem}\label{lemcons}
Let $C:H\to \mathbb{N}$ be a multiset, where $H$ is an abelian group.
Then $\chi(C):=\sum_{h\in H}C(h)\chi(h)=0$ for every $1 \ne \chi \in \widehat{H}$ if and only if $C$ is constant on $H$.
\end{lem}
One can see that $G$ and $\{1\}$ are spectral sets and tiles so in these cases the tile-spectral direction automatically holds.
Therefore, we may assume $A \subsetneq G$ and $B \subsetneq G$. Hence, there are
$1\ne \chi,\chi' \in \widehat{G}$ with $\chi(A)\ne 0$ and $\chi'(B)\ne 0$ by Lemma \ref{lemcons}. Then it follows from $\widehat{{\bf 1}}_A\cdot\widehat{{\bf 1}}_B=|G|\cdot\delta_{1}$ that $\chi(B) =0$ and $\chi'(A)=0$.
\begin{lem}\label{lemsubgroup}

Let $G$ be an abelian group, which is the direct sum of subgroups $B$ and $P$. Let $A \subset G$ such that $A\bigoplus B=G$. Then $A$ is spectral and its spectrum is $P$.
\end{lem}

\begin{proof}
Plainly, $|P|=|A|$. Thus it is enough to prove that for every element $\chi  $ of $\widehat{P}\setminus\{1\}$ we have $\chi(A)= 0$.

We uniquely write the elements of $G$ as $(x,y)$, where $x \in P$ and $y \in B$, since $P\bigoplus B=G$. Now we define $\chi_{(u,v)}\in \widehat{G}$ such as
$$\chi_{(u,v)}(x,y)=\chi_{u}(x)\chi_{v}(y) ~~~~(u,x\in P, ~y,v\in B).$$
In this case for every $u\ne u'$ are in $P$ we have
\begin{equation*}
\begin{split}
&\sum_{(x,y)\in A}\chi_{(u,0)}(x,y) \overline{\chi}_{(u',0)}(x,y) =\sum_{(x,y)\in A}\chi_{(u-u',0)}(x,y)=\\
&\sum_{(x,y)\in A}\chi_{u-u'}(x)1(y)=\sum_{x\in P}\chi_{(u-u',0)}(x,0)=0,
\end{split}
\end{equation*} where $x\in P, y\in B$ and $1\ne\chi_{u-u'}\in \widehat{P}, 1\in \widehat{B}$.
The last equality follows from Lemma \ref{lemcons}.
\end{proof}
Note that $\Z_p^2 \times \Z_q$ is a complemented group, which means that if $K$ is a subgroup of $G$, then there is a subgroup $L$ of $G$ with $K\bigoplus L=G$.
Thus Lemma \ref{lemsubgroup} implies the following.
\begin{cor}\label{corspect}
If $A$ is a subset of $G=\Z_p^2 \times \Z_q$ such that $A\bigoplus B=G$, where $B$ is a subgroup of $G$, then $A$ is spectral.
\end{cor}
\medskip
\noindent
{\bf Proof of tile-spectral direction of Theorem \ref{thm1}.}

Note that in our situation $G=\Z_p^2 \times \Z_q$.

\noindent
\underline{{\bf Case 1:} $|A|=q, ~|B|=p^2$} \\
Since $p\nmid|A|$, by Corollary \ref{cor1}, it follows that $\phi(A)\ne 0$ if $\phi\in \widehat{G}$ such that $o(\phi)=p$. We have already seen that $\phi(A)\phi(B)=0$ if $\phi \ne 1$. Thus $\phi(B)=0$ for every character of order $p$. Note that, it follows from this argument that $\Z_p^2$ is a spectrum for $B$ since $|B|=p^2=|\Z_p^2|$ and $B$ vanishes on each nonzero elements of $\widehat{\Z}_p^2$, so the elements of $(\chi_l)_{l\in \Z_p^2}$ are pairwise orthogonal. Thus $\widehat{\Z}_p^2$ is an orthogonal basis on $B$.

Similarly, $q\nmid |B|$ implies that $\chi(B)\ne 0$ for all $\chi\in \widehat{G}$ of order $q$. Hence $\chi(A)=0$ for all $\chi \in \widehat{G}$ with $o(\chi)=q$. In this case $\Z_q$ is a spectrum for $A$ since $|A|=q=|\Z_q|$ and hence $\widehat{\Z}_q$ is an orthogonal basis on $A$, as above.




\noindent
\underline{{\bf Case 2:} $|A|=p, ~|B|=pq$} \\
Now, we prove that $A$ is spectral.
First we show that there exists a character $\chi_{(a,0)}$ for some $a\in \Z_p^2\setminus\{0\}$ (which is of order $p$),
such that $\chi_{(a,0)}(A)=0$. As above, this immediately implies that the subgroup generated by $\chi_{(a,0)}$ is a spectrum for $A$.

By contradiction, we assume that $\chi_{(a,0)} (A)\ne 0$ for every $1\ne \chi_{(a,0)} \in \widehat{G}$. Hence $\chi_{(a,0)}  (B)=0$ for every $1\ne \chi_{(a,0)} \in \widehat{G}$. 
We can uniquely write the elements of $G$ as $(x,y)$, where $x \in \Z_p^2$ and $y \in \Z_q$. Let the multiset $\overline{B}$ be defined as $\overline{B}=\{(x, 0) \mid (x,y) \in B \}$. Then $\overline{B}$ is a constant function on $\Z_p^2$ by Lemma \ref{lemcons} so $p^2 \mid |\overline{B}|=|B|$. This is a contradiction, which implies the existence of $\chi_{(a,0)}$ of order $p$ satisfying  $\chi_{(a,0)}(A)=0$ for some $a\in \Z_p^2\setminus\{0\}$.

Note that, the same argument implies that there exists a character $\chi_{(a',0)}  $ for some $a'\ne a\in \Z_p^2\setminus\{0\}$
such that $\chi_{(a',0)}(B)=0$.

Now we show that $B$ is spectral.
Since $q\nmid |A|$ we have $\chi_{(0,b)}(B)=0$ for every $\chi_{(0,b)} \in \widehat{G}$ of order $q$. If $\psi(B)=0$ for every element of the group generated by $\chi_{(a',0)}$ and $\chi_{(0,b)}$ ($b \ne 0)$, then $B$ is spectral.


Therefore we may assume $\chi_{(c,d)}(B)\ne 0$  for some\footnote{If $\chi_{(c,d)}(B)\ne 0$ for some $c \ne 0$, $d \ne 0$, then $\chi_{(xc,yd)}(B)\ne 0$  with $x \in \Z_p\setminus\{0\}$, $y \in \Z_q\setminus\{0\}$.} $\chi_{(c,d)} \in \widehat{G}$ with $o(\chi_{(c,d)})=pq$. Then $\chi_{(c,d)}(A)=0$. 
Let $A_{(c,d)}=|\{(x,v) \in A \mid  \langle x-u,c\rangle=0 \}|$ denote (as in \eqref{eqproj}) the {\it projection} of $A$ to the subgroup isomorphic to $\Z_{pq}$ which is generated by
$(c,0)$ and $(0,d)$. By Lemma \ref{twoprimes}, $A_{(c,d)}$ is the sum of $\Z_p$-cosets and $\Z_q$-cosets. Plainly, $A_{(c,d)}$ is a $\Z_p$-coset since $|A|=p$. We may assume $0 \in A$. Thus, the set $A_{(c,d)}$ is a subgroup isomorphic to $\Z_p$.
This implies that $A_{(c,d)}$ consists of elements whose $q$-coordinate is $0$. It follows from the definition of $A_{(c,d)}$ that the same holds for the elements of $A$.
We obtain that $A$ is contained\footnote{$\Z_p^2$ embeds uniquely and naturally in $\Z_p^2 \times \Z_q$.} in $\Z_p^2\le G$.
Let $B_i= \{ (x,i) \in B \mid x \in \Z_p^2\}$, where $i \in \Z_q$.
Then for every $i \in \Z_q$ we have $A \bigoplus (B_i-i)=\Z_p^2$.


For $(a,0) \in A-A \subset \Z_p^2\times \Z_q$ let $a' \in \Z_p^2 $ be an element of $\Z_p^2$ satisfying $\langle a,a'\rangle =0$. Then $\chi_{(a',0)}(A)\ne 0$, otherwise $\chi_{(a',0)}(A)= 0$ would imply that $A_{(a',0)}$ is equidistributed on the cosets of the subgroup generated by $(a,0)$  and $(0,b)$ ($b \in \Z_q \setminus \{0\}$). As $|A|=p$, this would imply that any of these cosets contains exactly one element of $A$, contradicting the fact that there are two elements of $A$ whose difference is $(a,0)$ (i.e. they are contained in the same coset).

We may consider $A$ and  $(B_i-i)$ ($i \in \Z_q$) as subsets of $\Z_p^2$ and we identify $\chi_{(a',0)} \in \widehat{G}$ with $\chi_{a'}$ which is an element of $\widehat{\Z}_p^2$.
Since $A \bigoplus (B_i-i)=\Z_p^2$ for every $i \in \Z_q$ and $\chi_{a'}(A)\ne 0$, we have $\chi_{a'}(B_i-i)= 0$ for all $i \in \Z_q$.
This means that all $(B_i-i)$ as subsets of $\Z_p^2$
are equidistributed on the cosets of $\langle a \rangle$. Since $|B_i|=p$, the set $B_i\subset G$ has one element on each $\langle (a,0) \rangle$-coset. As $(a,0)$ is independent from the choice of $i$, we have $B \bigoplus \langle (a,0) \rangle=G$ so
Corollary \ref{corspect} gives that $B$ is spectral.
\qed

\section{Spectral-Tile}
The purpose of this section is to prove that every spectral set in $\Z_p^2 \times \Z_q$ tiles.

\medskip
\noindent
{\bf Proof of spectral-tile direction of Theorem \ref{thm1}.}

Let $(S,\La)$ be a spectral pair.
Assume $|S|= |\La|>1$, since if $|S|=1$, then $S$ is a tile.
 It is proved in \cite{Z_p^2, MalKol}  that Fuglede's conjecture holds for every proper subgroup of $\Z_p^2 \times \Z_q$.
 Hence we may assume neither $S$ nor $\La$ are contained in a proper subgroup of $\Z_p^2 \times \Z_q$, see Lemma 4.3 and Lemma 4.4 in \cite{mi}.

 We write that $m\mid\mid |S|$, if $gcd(|G|,|S|)=m$.
 We distinguish several cases as follows.
\noindent
\underline{{\bf Case 1:} $p^2\mid\mid |S|$} \\
In this case either $S$ is a tile or there are two elements of $S$ that only differ in their $q$-coordinates.
Then $\chi(\La)=0$ for some $o(\chi)=q$. By Corollary \ref{cor1}, $q\mid |\La|=|S|$, a contradiction.

\noindent
\underline{{\bf Case 2:} $pq\mid\mid |S|$} \\
If $p+1$ elements of $S$ are contained in a $\Z_{p}^2$-coset, then every nonzero element in $\Z_{p}^2$ has a nonzero multiple which is in $S-S$.
 In this case we say that {\it every direction of $\Z_p^2$ appears in $S$.}
Hence for all $1 \ne \chi\in \widehat{\Z}_p^2 \cong \widehat{H} \le \widehat{
G}$ we have $\chi(\La)=0$. This means that $\La$ vanishes in every nontrivial element of $\widehat{H}$, so does the projection $\overline{\La}=\{(x,0)\mid (x,y)\in \La \}$ of $\La$ onto $\Z_p^2$. Then by Lemma \ref{lemcons}, it follows that $p^2\mid |\overline{\La}|=|\La|$, which contradicts our assumption.

Therefore we can assume that $S$ has exactly $p$ elements in every $\Z_p^2$-coset. Further this argument implies that there is an $a \in \Z_p^2$ such that the nonzero multiples of $(a,0)$ do not appear in the set $S-S$.
Since $S=pq$, this implies that each $\langle( a,0) \rangle$-coset contains exactly one element of $S$, which shows that $S$ tiles $G$ with tiling complement $\langle (a,0) \rangle$.

\noindent
\underline{{\bf Case 3:} $pq\ge |S|$} \\
Following the argument of the previous case we have that $p^2 \mid |S|$ or $S$ is a tile. The former case is handled in Case 1.

\noindent
\underline{{\bf Case 4:} $1\mid\mid |S|$ or $q\mid\mid |S|$} \\
As $p\nmid |S|$ there cannot be two elements of $S$ in the same $\Z_p^2$-coset. Then either $S$ has $q$ elements, and it is a tile with tiling complement $\Z_p^2$ or $|S|=|\La|<q$. Hence we may assume $1\mid\mid |S|$.
By Corollary \ref{cor1}, we have $\chi(S)\ne 0$, if $o( \chi)\in \{1,p,q\}$.
As $|\La| >1$, there exists a $\chi_{(a,b)}$ for some $(a,b) \in \La-\La$ with $o(\chi_{(a,b)})=pq$ (i.e. $a \ne 0$, $b \ne 0$) satisfying $\chi_{(a,b)}(S)=0$.
By Lemma \ref{twoprimes}, $S_{(a,b)}$ (the projection of $S$ defined by \eqref{eqproj}) is the sum of $\Z_p$-cosets and $\Z_q$-cosets. This implies that $|S|=|S_{(a,b)}| = kp+lq$, where $k,l\ge 0$. As $1<|S|<q$ we have $l=0$ and then $p\mid |S|$, a contradiction.

\noindent
\underline{{\bf Case 5:} $p\mid\mid |S|$ and $|S|<pq$} \\
Now we assume that there exists a spectral pair $(S,\La)$ such that $p\mid\mid |S|$, $1<|S|<pq$ and neither $S$ nor $\La$ is contained in any proper subgroup of $G$.
We show that this leads to a contradiction.

If $\La$ is not contained in a $\Z_p^2$-coset or a $\Z_q$-coset, then there is $\chi_{(a,b)} \in \widehat{G}$ ($0\ne a \in \Z_p^2$, $0 \ne b \in \Z_q$) such that $\chi_{(a,b)}(S)=0$ and $o(\chi_{(a,b)})=pq$.

By applying Lemma \ref{twoprimes} for $S_{(a,b)}$ we obtain that $|S_{(a,b)}|=|S|=kp+lq$ with $k, l \in \N$. Then $l=0$, since $p \mid |S|$ and $|S|<pq$. Therefore, $S_{(a,b)}$ is a sum of $\Z_p$-cosets.
Thus, we have $a_i \cdot p$ elements of $S_{(a,b)}$ in the coset $i+\Z_{p}^2$ for $i= 0 \stb q-1$.
The multiset $S_{(a,b)}$ is considered as a projection of $S$ and it is easy to see that an element and its image have the same $q$-coordinate.
Finally, we obtain that $a_i \cdot p$ elements of $S$ are in the coset $i+\Z_{p}^2$ for $i= 0 \stb q-1$.

Similar argument holds if we change the role of $S$ and $\La$ with possibly different sequence of $b_i$'s satisfying that $b_i \cdot p$ elements of $\La$ are contained in the coset $i+\Z_p^2$ for $i=0 \stb q-1$.

If $a_i>1$ (resp. $b_i>1$) for some $i=0 \stb q-1$, then we have at least $2p>p$ elements of $S$ (resp. $\La$) in a $\Z_p^2$-coset. Then every direction of $\Z_p^2$ appears in $S$ (resp. $\La$).
Thus we may apply Lemma \ref{lemcons} to $\overline{\La}=\{(x,0)\mid (x,y)\in \La \}$ (resp. $\overline{S}$), so as above we have $p^2 \mid |S|=|\La|$, a contradiction.
From now on we assume that every $\Z_p^2$-coset contains either $0$ or $p$ elements of $S$ and $\La$.

Let $k$ denote the number of $\Z_p^2$-cosets containing $p$ elements of $S$, whence $|S|=k\cdot p$.  If $k=1$, then $S$ is contained in exactly one $\Z_p^2$-coset. By the assumption that $0\in S$ this means that $S$ contained in a proper subgroup of $S$ which is excluded, hence we can assume that $k \ge 2$.
Moreover, $2 \le k < q$, since  $|S|<pq$.
Note that $\overline{S}$ is a set, since $q \nmid |S|$. Otherwise, $S-S$ contains an element of order $q$ and the fact that $(\La,S)$ is also a spectral pair implies $q \mid |\La|=|S|$. Further  $|\overline{S}|>p$ so every direction of $\Z_p^2$ appears in $\overline{S}$.

We claim that if $S-S$ contains an element $(a,0)$ of order $p$ ($a \ne 0$), then $S-S$ does not contain $(ca,b)\in G$, where $b \in \Z_q\setminus \{0\}$, $c \in \Z_p \setminus\{0\}$, and vice versa, if $S-S$ contains an element $(a,b)$ of order $pq$ ($a \ne 0, b\ne 0$), then $S-S$ does not contain $(ca,0)$, where $c \in \Z_p \setminus\{0\}$). The same holds for $\La-\La$.
Indeed, suppose $(a,0) \in \Z_p^2 \cap (S-S)$.  If $\langle a , a' \rangle=0$ for some $0 \ne a' \in \Z_p^2 $, then $\La-\La$ cannot contain $(c'a',x)$ with $0 \ne c' \in \Z_p ,0 \ne x \in \Z_q$, since otherwise $\chi_{(c'a',x)}(S)=0$ that we exclude now. It follows from Lemma \ref{twoprimes} that $S_{(c'a',x)}(=S_{(a',x)})$ is the sum of $\Z_p$-cosets and $\Z_q$-cosets so $|S|=kp+lq$. Since $p \mid |S|$
and $|S|<pq$, we have $l=0$ so $S_{(c'a',x)}$ as a multiset is the sum of $\Z_p$-cosets only. We have already assumed that each $\Z_p^2$-coset contains at most $p$ elements of $S$. Therefore, $S_{(c'a',x)}$ is a set. This contradicts the fact that $S-S$ contains $(a,0)$ so $(c'a',x) \not\in \La-\La$. 

We have seen $\overline{\La}$ contains every direction of $\Z_p^2$ and $(c'a',x) \not\in \La-\La$ for any $0 \ne c' \in \Z_p ,0 \ne x \in \Z_q$, hence $(c''a',0) \in \La-\La$ for some $0 \ne c'' \in \Z_p$. Repeating the argument above for the spectral pair $(\La,S)$ we obtain that $(ca,b) \not\in S-S$ for every $c \in \Z_p \setminus\{0\}, b \in \Z_q \setminus \{ 0\}$. This implies the statement that we claimed for $S-S$. Analogue argument verifies the statement for $\La-\La$.
Note that in addition, the argument above implies that if $(a,0)\in S-S$ ($a \ne 0$), then there is a $0 \ne c'' \in \Z_p$ such that $(c''a',0)\in \La-\La$, where $\langle a , a'\rangle=0$. The same holds by changing the role of $S$ and $\La$.

Now we take two elements of $\La$ in a $\Z_p^2$-coset, whose difference is $(a', 0)$, where $0\ne a' \in \Z_p^2$. Then since $(S,\La)$ is a spectral pair, for $0 \ne a  \in \Z_p^2$ with $\langle a , a'\rangle=0$ we have that $\bar{S}$ is equidistributed on the $\langle a \rangle$-cosets. Thus each $\langle a \rangle$-coset contains $k$ elements of $\overline{S}\subset \Z_p^2$.
Since $(a',0)\in \La-\La$, by the last result of the previous paragraph, we get that for any $0\ne x\in \Z_p^2$ the difference of preimages of any two elements of $\overline{S} \cap (\langle a \rangle+x)$ is of the form $(ca,0)$ for some $0 \ne c \in \Z_p$.
Thus the preimages of the elements of $\overline{S} \cap (\langle a \rangle+x) \in G$ are contained in a $\langle (a,0) \rangle$-coset. Since there is a $\Z_p^2$-coset containing exactly $p$ elements of $S$ and each $(a,0)$-coset contains either $k$ or $0$ elements of $S$, we have $k \mid p$. Since $k \ge 2$, it follows that $k=p$. Then $|S|=p^2$, which contradicts our assumption that $p\mid\mid |S|$.
\qed

\section*{Acknowledgement} The authors would like to express their gratitude to the referee for the valuable comments that essentially influence the quality of the paper.

Research by the first author was supported by a Premium Postdoctoral Fellowship of the Hungarian Academy of Science, and by NKFIH (Hungarian National Research, Development
and Innovation Office) grant K-124749.

The second author is supported by the J\'anos Bolyai Research Scholarship of the Hungarian Academy of Sciences and by the \'UNKP-20-5 New National Excellence Program of the Ministry for Innovation and
Technology from the source of the National Research, Development and Innovation Fund (NKFIF) and NKFIH SNN 132625.

Support provided from the NKFIF of Hungary, financed under the Thematic Excellence
Programme no.
2020-4.1.1.-TKP2020 (National Challenges Subprogramme) funding scheme.

\end{document}